\documentclass[a4paper,draft]{amsproc}
\usepackage{amssymb}
\usepackage[hyphens]{url} 

\theoremstyle{plain}
 \newtheorem{theorem}{Theorem}[section]
 \newtheorem{proposition}{Proposition}[section]
 \newtheorem{lemma}{Lemma}[section]
 \newtheorem{corollary}{Corollary}[section]
\theoremstyle{definition}

\theoremstyle{remark}

 \numberwithin{equation}{section}

\renewcommand{\leq}{\leqslant}
\renewcommand{\geq}{\geqslant}

\title[ Dedekind sum/ M. Goubi]{New arithmetical proof of the reciprocity law for Dedekind sums}



\author{\bfseries Mouloud  Goubi} 
\address{Mouloud Goubi\\
Department of Mathematics \\
University of UMMTO, P.O.Box 17,RP 15000\\
Tizi-ouzou, Algeria\\
Laboratoire d'Alg\`ebre et Th\'eorie des Nombres, USTHB Alger}
\email{mouloud.goubi@ummto.dz}

\begin{document}

\vspace{18mm} \setcounter{page}{1} \thispagestyle{empty}

\begin{abstract}
In this paper,  for coprime numbers $p$ and $q$ we consider the \\
Dedekind sums
\begin{eqnarray}\label{ded-sums}
S\left(p,q\right)=\sum_{r=1}^{q-1}\left\{\frac{r}{q}\right\}\left\{\frac{rp}{q}\right\}.
\end{eqnarray}
 First, we give an improvement of the proof given by H. Rademacher and A. Whiteman
 \cite{RAD1}, and we construct a new arithmetical proof for the reciprocity law
\begin{eqnarray}\label{reci-law}
S\left(p,q\right)+S\left(q,p\right)=\frac{p^2+q^2+1}{12pq}+\frac{p+q}{4}-\frac{3}{4}.
\end{eqnarray}
different of all the arithmetical proofs given until now.\\

Second, we found explicit formula of $S\left(p,q\right)$ for
$q\equiv 1,p-1,2,p-2,3,p-3,p-4$ and $4[p]$.
\end{abstract}

\maketitle

\section{Introduction and statement of main results}
\subsection{Introduction}
 In the literature there are several different proofs of the
 reciprocity law for Dedekind-Rademacher sums, H. Rademacher and E.
 Grosswald (in \cite{RAD2}) have constructed four proofs.\\
 In this work, we are interested by the arithmetical ones. First we give
 an improvement of the proof of H. Rademacher and A. Whiteman
\cite[\S 3]{RAD1}. In the second time, using Euclidean division, we
give a new arithmetical proof of such reciprocity law.\\
Taking $q\equiv b(p)$, the idea of the proof consists to write
$$S(p,q)+S(q,p)=P(b),$$ where $P$ is a polynomial of degree $3$.
After we establish that $P$ is a constant polynomial, and
$$P(b)=\frac{p^2+q^2+1}{12pq}+\frac{p+q}{4}-\frac{3}{4}$$
\\

 Finally from the reciprocity law and the expression of
 $S\left(p,q\right)$ we found explicit formula of $S\left(p,q\right)$ for $q\equiv 1,p-1,2,p-2,3,p-3,p-4$ and $4[p]$.\\

 In this work, we need the following two well-known results for
 finite sums
 \begin{equation}\label{sum1}
 \sum_{r=1}^{q-1}r=\frac{q\left(q-1\right)}{2},
 \end{equation}
  \begin{equation}\label{sum2}
 \sum_{r=1}^{q-1}r^2=\frac{q\left(q-1\right)\left(2q-1\right)}{6}
 \end{equation}
which can be proven by recursion.
\subsection{Statement of main results}
Let the first normalized Bernoulli function
\begin{eqnarray}\label{chi}
B_1(t):= \left\{
\begin{array}{ccc}
\left\{t\right\}-\frac{1}{2} ,&\quad \textrm{ if }\ t\in \mathbb{N}, \\
0\ , &\quad  \textrm{ otherwise}.
\end{array}
\right.
\end{eqnarray}
Where $\left\{t\right\}=t-\left\lfloor t\right\rfloor$, and
$\left\lfloor t\right\rfloor$ is the greater integer less then
$t$.\\

For $p,q$ two coprime numbers, where $q$ is any integer, and $p$, is
of course a positive integer consider the Dedekind sums
\begin{eqnarray}\label{spq}
s\left(p,q\right)=\sum_{r=1}^{q-1}B_1\left(\frac{r}{q}\right)B_1\left(\frac{rp}{q}\right)
\end{eqnarray}
and
\begin{eqnarray}\label{Spq}
S\left(p,q\right)=\sum_{r=1}^{q-1}\left\{\frac{r}{q}\right\}\left\{\frac{rp}{q}\right\}
\end{eqnarray}
If the class modulo $p$ of $q$ is $b$ $(1\leq b\leq p-1)$, then
$$s\left(q,p\right)=s\left(b,q\right)$$ and
$$S\left(q,p\right)=S\left(b,p\right).$$\\
Without losing generality only we consider in this work $p<q$ and
$p,q$ coprime. In this case
$$\sum_{r=1}^{q-1}B_1\left(\frac{rp}{q}\right)=0$$
then $s\left(p,q\right)$ can be transformed to
\begin{eqnarray*}
s\left(p,q\right)&=&\sum_{r=1}^{q-1}\left(\left\{\frac{r}{q}\right\}-\frac{1}{2}\right)B_1\left(\frac{rp}{q}\right)\\
&=&\sum_{r=1}^{q-1}\left\{\frac{r}{q}\right\}B_1\left(\frac{rp}{q}\right)
\end{eqnarray*}
\begin{theorem}\label{premiertheo}
For $p,q$ positive coprime numbers, we have
\begin{eqnarray}\label{loiricipro}
S\left(p,q\right)+S\left(q,p\right)=\frac{p^2+q^2+1}{12pq}+\frac{p+q}{4}-\frac{3}{4}
\end{eqnarray}
\end{theorem}
The reciprocity law in \eqref{loiricipro} involves the reciprocity
law of $s\left(p,q\right)$:
\begin{eqnarray}\label{loiriciproprin}
s\left(p,q\right)+s\left(q,p\right)=\frac{p^2+q^2-3pq+1}{12pq}
\end{eqnarray}

Specifically in the case $q\equiv 1,2,3$ and $4[p]$, we obtain

\begin{theorem}\label{deuxiemetheo}
For $p<q$ positive coprime numbers, we have
\begin{eqnarray}\label{loiricipro1}
q\equiv1[p]~,~S\left(p,q\right)=\frac{p^2+q^2+1}{12pq}+\frac{q}{4}-\frac{p}{12}-\frac{1}{6p}-\frac{1}{4}
\end{eqnarray}
\begin{eqnarray}\label{equa1}
q\equiv2[p]~,~S\left(p,q\right)=\frac{p^2+q^2+1}{12pq}+\frac{q}{4}-\frac{p}{24}-\frac{5}{24p}-\frac{1}{4}
\end{eqnarray}
\begin{eqnarray}\label{equa2}
q\equiv3[p]~,~S\left(p,q\right)=\frac{p^2+q^2+1}{12pq}+\frac{q}{4}-\frac{p}{36}-\frac{5}{18p}-\frac{1}{3}\left\{\frac{p}{3}\right\}-\frac{1}{12}
\end{eqnarray}
\begin{eqnarray}\label{equa3}
q\equiv4[p]~,~S\left(p,q\right)=\frac{p^2+q^2+1}{12pq}+\frac{q}{4}-\frac{p}{48}-\frac{17}{48p}-\frac{1}{2}\left\{\frac{p}{4}\right\}
\end{eqnarray}
\end{theorem}
As a consequence we deduce  for $q\equiv p-1, p-2, p-3$ and $p-4$
modulo $p$ that
\begin{corollary}\label{coro1}
For $p<q$ positive coprime numbers, we have
\begin{eqnarray}\label{equa4}
q\equiv(p-1)[p]~,~S\left(q-p,q\right)&=&\frac{q}{4}-\frac{p^2+q^2+1}{12pq}+\frac{p}{12}+\frac{1}{6p}-\frac{1}{4}
\end{eqnarray}
\begin{eqnarray}\label{equa5}
q\equiv(p-2)[p]~,~S\left(q-p,q\right)=\frac{q}{4}-\frac{p^2+q^2+1}{12pq}+\frac{p}{24}+\frac{5}{24p}-\frac{1}{4}
\end{eqnarray}
\begin{eqnarray}\label{equa6}\\
\nonumber
q\equiv(p-3)[p]~,~S\left(q-p,q\right)=\frac{q}{4}-\frac{p^2+q^2+1}{12pq}+\frac{p}{36}+\frac{5}{18p}+\frac{1}{3}\left\{\frac{p}{3}\right\}-\frac{5}{12}
\end{eqnarray}
\begin{eqnarray}\label{equa7}\\
\nonumber
q\equiv(p-4)[p]~,~S\left(q-p,q\right)=\frac{q}{4}-\frac{p^2+q^2+1}{12pq}+\frac{p}{48}+\frac{17}{48p}+\frac{1}{2}\left\{\frac{p}{4}\right\}-\frac{1}{2}
\end{eqnarray}
\end{corollary}
\section{Improvement of the short proof of Rademacher and Whiteman}
Here we give an improvement of the short proof of Rademacher and
Whiteman \cite{RAD1} different from the proof given by L. J.
Mordell. \cite{MORD}. To do this we need the following lemma.
\begin{lemma}\label{lemma1}
For $p,q$ two coprime numbers, we have
\begin{equation}\label{equa8}
\sum_{r=1}^{q-1}\sum_{t<\frac{rp}{q}}t=\frac{\left(p-1\right)\left(2pq-3p+2q\right)}{12}+ps\left(q,p\right)
\end{equation}
\end{lemma}
\begin{proof}
Since $(p,q)=1$, $\frac{tq}{p}$ is not an integer for $1\leq t\leq
p-1$. Then $r<\frac{tq}{p}$ means that
$r\leq\left\lfloor\frac{tq}{p}\right\rfloor$ and then
\begin{eqnarray*}
\sum_{r=1}^{q-1}\sum_{t<\frac{rp}{q}}t&=&\sum_{t=1}^{p-1}\sum_{r>\frac{tq}{p}}t\\
&=&\sum_{t=1}^{p-1}\left(\sum_{r=1}^{q-1}t-\sum_{r<\frac{tq}{p}}t\right)\\
&=&\sum_{t=1}^{p-1}\sum_{r=1}^{q-1}t-\sum_{t=1}^{p-1}\sum_{r<\frac{tq}{p}}t\\
&=&\frac{\left(q-1\right)\left(p-1\right)p}{2}-\sum_{t=1}^{p-1}t\left\lfloor\frac{tq}{p}\right\rfloor\\
&=&\frac{\left(q-1\right)\left(p-1\right)p}{2}-\sum_{t=1}^{p-1}t\left(\frac{tq}{p}-\left\{\frac{tq}{p}\right\}\right)\\
&=&\frac{\left(q-1\right)\left(p-1\right)p}{2}-\frac{q}{p}\sum_{t=1}^{p-1}t^2+p\sum_{t=1}^{p-1}\left\{\frac{t}{p}\right\}\left\{\frac{tq}{p}\right\}\\
&=&\frac{\left(q-1\right)\left(p-1\right)p}{2}-\frac{q\left(p-1\right)\left(2p-1\right)}{6}+pS\left(q,p\right)\\
&=&\frac{\left(q-1\right)\left(p-1\right)p}{2}-\frac{q\left(p-1\right)\left(2p-1\right)}{6}+\frac{p\left(p-1\right)}{4}+ps\left(q,p\right)\\
&=&\frac{\left(2q-1\right)\left(p-1\right)p}{4}-\frac{q\left(p-1\right)\left(2p-1\right)}{6}+ps\left(q,p\right)\\
&=&\frac{\left(p-1\right)\left(2pq-3p+2q\right)}{12}+ps\left(q,p\right)
\end{eqnarray*}
\end{proof}
We compute the value of the sum
$\sum_{r=1}^{q-1}B_1\left(\frac{rp}{q}\right)^2$ with two different
methods, and the comparison of the results gives
the proof of the reciprocity law.\\
In one hand we have
\begin{equation*}
\sum_{r=1}^{q-1}B_1\left(\frac{rp}{q}\right)^2=\sum_{r=1}^{q-1}B_1\left(\frac{r}{q}\right)^2=\sum_{r=1}^{q-1}\left(\frac{r}{q}-\frac{1}{2}\right)^2=
\frac{1}{q^2}\sum_{r=1}^{q-1}r^2-\frac{1}{q}\sum_{r=1}^{q-1}r+\frac{1}{4}\sum_{r=1}^{q-1}1
\end{equation*}
Then
\begin{equation*}
\sum_{r=1}^{q-1}B_1\left(\frac{rp}{q}\right)^2=\frac{\left(q-1\right)\left(2q-1\right)}{6q}-\frac{q-1}{2}+\frac{q-1}{4}
\end{equation*}
Furthermore
\begin{equation*}
\sum_{r=1}^{q-1}B_1\left(\frac{rp}{q}\right)^2=\frac{\left(q-1\right)\left(2q-1\right)}{6q}-\frac{q-1}{4}
\end{equation*}
In other hand
\begin{equation*}
\sum_{r=1}^{q-1}B_1\left(\frac{rp}{q}\right)^2=\sum_{r=1}^{q-1}\left(\left\{\frac{rp}{q}\right\}-\frac{1}{2}\right)^2
=\sum_{r=1}^{q-1}\left(\frac{rp}{q}-\left\lfloor\frac{rp}{q}\right\rfloor-\frac{1}{2}\right)^2
\end{equation*}
Thus
\begin{equation*}
\sum_{r=1}^{q-1}B_1\left(\frac{rp}{q}\right)^2=\sum_{r=1}^{q-1}\left(\left(\frac{rp}{q}\right)^2-\frac{2rp}{q}\left(\left\lfloor\frac{rp}{q}\right\rfloor+\frac{1}{2}\right)
+\left(\left\lfloor\frac{rp}{q}\right\rfloor+\frac{1}{2}\right)^2\right)
\end{equation*}
Then
\begin{equation*}
\sum_{r=1}^{q-1}B_1\left(\frac{rp}{q}\right)^2=-\frac{p^2}{q^2}\sum_{r=1}^{q-1}r^2+2p\sum_{r=1}^{q-1}\frac{r}{q}B_1\left(\frac{rp}{q}\right)+
\sum_{r=1}^{q-1}\left(\left\lfloor\frac{rp}{q}\right\rfloor+\frac{1}{2}\right)^2
\end{equation*}
and
\begin{equation*}
\sum_{r=1}^{q-1}B_1\left(\frac{rp}{q}\right)^2=-\frac{p^2}{q^2}\sum_{r=1}^{q-1}r^2+2ps\left(p,q\right)+
\sum_{r=1}^{q-1}\left\lfloor\frac{rp}{q}\right\rfloor\left(\left\lfloor\frac{rp}{q}\right\rfloor+1\right)+\frac{q-1}{4}
\end{equation*}
But
\begin{equation*}
\left\lfloor\frac{rp}{q}\right\rfloor\left(\left\lfloor\frac{rp}{q}\right\rfloor+1\right)=2\sum_{t<\frac{rp}{q}}t
\end{equation*}
then
\begin{equation*}
\sum_{r=1}^{q-1}B_1\left(\frac{rp}{q}\right)^2=\frac{q-1}{4}-\frac{p^2\left(q-1\right)\left(2q-1\right)}{6q}+2ps\left(p,q\right)+
2\sum_{r=1}^{q-1}\sum_{t<\frac{rp}{q}}t
\end{equation*}
and from the value of $\sum_{r=1}^{q-1}\sum_{t<\frac{rp}{q}}t$ in
Lemma \eqref{lemma1} we deduce that
\begin{eqnarray*}
\sum_{r=1}^{q-1}B_1\left(\frac{rp}{q}\right)^2&=&\frac{q-1}{4}-\frac{p^2\left(q-1\right)\left(2q-1\right)}{6q}+2ps\left(p,q\right)\\
&+&\frac{\left(p-1\right)\left(2pq-3p+2q\right)}{6}+2ps\left(q,p\right)
\end{eqnarray*}
and then
\begin{eqnarray*}
2p\left(s\left(p,q\right)+s\left(q,p\right)\right)&=&-\frac{q-1}{4}+\frac{p^2\left(q-1\right)\left(2q-1\right)}{6q}\\
&-&\frac{\left(p-1\right)\left(2pq-3p+2q\right)}{6}+\frac{\left(q-1\right)\left(2q-1\right)}{6q}-\frac{q-1}{4}
\end{eqnarray*}
Finally
\begin{eqnarray*}
2p\left(s\left(p,q\right)+s\left(q,p\right)\right)=\frac{1}{6q}\left(p^2+q^2-3pq+1\right)
\end{eqnarray*}
Thus
\begin{eqnarray*}
s\left(p,q\right)+s\left(q,p\right)=\frac{p^2+q^2-3pq+1}{12pq}
\end{eqnarray*}

\section{Proof of Theorem \ref{premiertheo} and Theorem \ref{deuxiemetheo}}
We start this section with some useful preliminaries results.

\subsection{ finite sums involving fractional part function}
\begin{lemma}\label{lemma2}
\begin{equation}\label{equa9}
\sum_{r=1}^{q-1}\left\{\frac{rp}{q}\right\}=\frac{q-1}{2}
\end{equation}
\begin{equation}\label{equa10}
\sum_{r=1}^{q-1}\left\{\frac{rp}{q}\right\}^2=\frac{\left(q-1\right)\left(2q-1\right)}{6q}
\end{equation}
\begin{equation}\label{equa11}
\sum_{r=1}^{q-1}\left\lfloor\frac{rp}{q}\right\rfloor=\frac{\left(p-1\right)\left(q-1\right)}{2}
\end{equation}
\begin{equation}\label{equa12}
\sum_{r=1}^{q-1}\left\lfloor\frac{rp}{q}\right\rfloor^2=\frac{\left(p^2+1\right)\left(q-1\right)\left(2q-1\right)}{6q}-2pS\left(p,q\right)
\end{equation}
\end{lemma}
\begin{proof}
For the first relation \eqref{equa9}, let $\bar{p}$ the inverse
modulo $q$ of $p$ then
\begin{equation*}
\sum_{r=1}^{q-1}\left\{\frac{r}{q}\right\}=\sum_{r=1}^{q-1}\left\{\frac{r\bar{p}p}{q}\right\}=\sum_{t=1}^{q-1}\left\{\frac{tp}{q}\right\}
\end{equation*}
and
\begin{equation*}
\sum_{r=1}^{q-1}\left\{\frac{rp}{q}\right\}=\frac{1}{q}\sum_{r=1}^{q-1}r=\frac{q-1}{2},
\end{equation*}
For the second relation \eqref{equa10}, we have
\begin{equation*}
\sum_{r=1}^{q-1}\left\{\frac{rp}{q}\right\}^2=\sum_{r=1}^{q-1}\left\{\frac{r}{q}\right\}^2=\frac{1}{q^2}\sum_{r=1}^{q-1}r^2=
\frac{\left(q-1\right)\left(2q-1\right)}{6q}.
\end{equation*}
The proof of the relation \eqref{equa11} is
\begin{eqnarray*}
\sum_{r=1}^{q-1}\left\lfloor\frac{rp}{q}\right\rfloor&=&\sum_{r=1}^{q-1}\left(\frac{rp}{q}-\left\{\frac{rp}{q}\right\}\right)\\
&=&\frac{p}{q}\sum_{r=1}^{q-1}r-\sum_{r=1}^{q-1}\left\{\frac{rp}{q}\right\}\\
&=&\frac{p\left(q-1\right)}{2}-\frac{q-1}{2}\\
&=&\frac{\left(p-1\right)\left(q-1\right)}{2}
\end{eqnarray*}
Finally for the relation \eqref{equa12} we have
\begin{eqnarray*}
\sum_{r=1}^{q-1}\left\lfloor\frac{rp}{q}\right\rfloor^2&=&\sum_{r=1}^{q-1}\left(\frac{rp}{q}-\left\{\frac{rp}{q}\right\}\right)^2\\
&=&\sum_{r=1}^{q-1}\left(\frac{r^2p^2}{q^2}-2\frac{rp}{q}\left\{\frac{rp}{q}\right\}+\left\{\frac{rp}{q}\right\}^2\right)\\
&=&\frac{p^2}{q^2}\sum_{r=1}^{q-1}r^2+\sum_{r=1}^{q-1}\left\{\frac{rp}{q}\right\}^2-2pS\left(p,q\right)\\
&=&\frac{p^2\left(q-1\right)\left(2q-1\right)}{6q}+\frac{\left(q-1\right)\left(2q-1\right)}{6q}-2pS\left(p,q\right)\\
&=&\frac{\left(p^2+1\right)\left(q-1\right)\left(2q-1\right)}{6q}-2pS\left(p,q\right)
\end{eqnarray*}
\end{proof}
\begin{corollary}\label{coro2}
\begin{eqnarray}\label{digamma1}
S\left(1,q\right)=\frac{\left(q-1\right)\left(2q-1\right)}{6q},
\end{eqnarray}
\begin{eqnarray}\label{digamma2}
s\left(p,q\right)=S\left(p,q\right)-\frac{q-1}{4}.
\end{eqnarray}
\end{corollary}
\begin{proof}
Since
$$ S\left(1,q\right)=\sum_{r=1}^{q-1}\left\{\frac{r}{q}\right\}^2$$
then
$$S\left(1,q\right)=\frac{1}{q^2}\sum_{r=1}^{q-1}r^2.$$ From the
relation \eqref{sum2} we deduce the result \eqref{digamma1}.
\begin{eqnarray*}
s\left(p,q\right)&=&\sum_{r=1}^{q-1}B_1\left(\frac{r}{q}\right)B_1\left(\frac{rp}{q}\right)\\
&=&\sum_{r=1}^{q-1}\left\{\frac{r}{q}\right\}\left(\left\{\frac{rp}{q}\right\}-\frac{1}{2}\right)\\
&=&S\left(p,q\right)-\frac{1}{2}\sum_{r=1}^{q-1}\left\{\frac{rp}{q}\right\}
\end{eqnarray*}
From the relation \eqref{equa9} Lemma \ref{lemma2} we deduce that
$$s\left(p,q\right)=S\left(p,q\right)-\frac{q-1}{4}.$$
\end{proof}
\subsection{Some properties of the Dedekind sums $S(p,q)$}

\begin{lemma}\label{lemma3}
For $p,q$ coprime such that $p<q$, we have
\begin{equation}\label{equa13}
S\left(q-p,q\right)=\frac{q-1}{2}-S(p,q)
\end{equation}
\begin{proof}
Using the well known property of the fractional part function for
any real $x$:
$$\left\{x\right\}+\left\{-x\right\}=1$$  we deduce
that
\begin{eqnarray*}
S\left(q-p,p\right)&=&\sum_{r=1}^{q-1}\left\{\frac{r}{q}\right\}\left\{\frac{r(q-p)}{q}\right\}\\
&=&\sum_{r=1}^{q-1}\left\{\frac{r}{q}\right\}\left\{\frac{-rp}{q}\right\}\\
&=&\sum_{r=1}^{q-1}\left\{\frac{r}{q}\right\}\left(1-\left\{\frac{rp}{q}\right\}\right)\\
&=&\sum_{r=1}^{q-1}\left\{\frac{r}{q}\right\}-S\left(p,q\right)\\
\end{eqnarray*}
From the relation \eqref{equa9} lemma \ref{lemma2} we deduce that
$$S\left(q-p,q\right)=\frac{q-1}{2}-S(p,q)$$
\end{proof}
\end{lemma}
The following proposition gives a new expression of
$S\left(p,q\right)$ as a sum of three quantities.
\begin{proposition}\label{propo1}
For $q>p$ positive coprime numbers, The Euclidean division of $q$
over $p$ gives $q=ap+b$ with $1\leq b\leq p-1$. Then we have
\begin{eqnarray}\label{equa14}
S\left(p,q\right)=\frac{1}{q^2}\sum_{n=0}^{p-1}\sum_{t=1}^{a}\left(an+t\right)\left(pt-n\right);~\text{if}~b=1
\end{eqnarray}
and
\begin{eqnarray}\label{equa15}\\
\nonumber
S(p,q)=\frac{1}{q^2}\sum_{n=0}^{p-1}\sum_{t=1}^{a}\left(an+t\right)
\left(tp-nb\right)+\frac{1}{q}\sum_{n=0}^{p-1}\sum_{t<\frac{nb}{p}}\left(an+t\right)+\frac{1}{q^2}\sum_{t=1}^{b-1}\left(ap+t\right)\left(q+pt-pb\right);~
\text{if}~b\geq2
\end{eqnarray}
\end{proposition}
\begin{proof}
$$S\left(p,q\right)=\sum_{r=1}^{q-1}\left\{\frac{r}{q}\right\}\left\{\frac{rp}{q}\right\}$$
Since $r$ lies to the set $\left\{1,2,3,...,q-1\right\}$ and
$q=ap+b$ we can write $r=an+t$ with $0\leq n\leq p-1$ and $1\leq
t\leq a$ for $r$ between $1$ and $q-b$. And $r=ap+t$ for $1\leq
t\leq b-1$ when $r$ lies to $\left\{ap+1,..., ap+b\right\}$. Then we distingue two cases\\

{\bf b=1}: in this case $r$ lies to the set
$\left\{1,2,...,ap\right\}$ and then
\begin{eqnarray*}
S\left(p,q\right)&=&\sum_{r=1}^{q-1}\left\{\frac{r}{q}\right\}\left\{\frac{rp}{q}\right\}\\
&=&\sum_{n=0}^{p-1}\sum_{t=1}^{a}\left\{\frac{an+t}{q}\right\}\left\{\frac{p\left(an+t\right)}{q}\right\}
\end{eqnarray*}
One remarks that
$$p\left(an+t\right)=\left(ap+1\right)n+pt-n=qn+pt-n,$$
$$1<pt-n\leq q-1,$$
and
$$1<an+t\leq q-1.$$
Then
$$\left\{\frac{p\left(an+t\right)}{q}\right\}=\left\{\frac{pt-n}{q}\right\}=\frac{pt-n}{q}$$
and we obtain
$$S\left(p,q\right)=\frac{1}{q^2}\sum_{n=0}^{p-1}\sum_{t=1}^{a}\left(an+t\right)\left(pt-n\right)$$\\

{\bf $b\geq 2$}: In this case we have
\begin{eqnarray*}
S\left(p,q\right)&=&\sum_{r=1}^{q-1}\left\{\frac{r}{q}\right\}\left\{\frac{rp}{q}\right\}\\
&=&\sum_{r=1}^{ap}\left\{\frac{r}{q}\right\}\left\{\frac{rp}{q}\right\}+\sum_{r=ap+1}^{ap+b-1}\left\{\frac{r}{q}\right\}\left\{\frac{rp}{q}\right\}\\
&=&\sum_{n=0}^{p-1}\sum_{t=1}^{a}\left\{\frac{an+t}{q}\right\}\left\{\frac{p\left(an+t\right)}{q}\right\}
+\sum_{r=1}^{b-1}\left\{\frac{ap+t}{q}\right\}\left\{\frac{\left(ap+t\right)p}{q}\right\}
\end{eqnarray*}
Remark that
$$p\left(an+t\right)=\left(ap+b\right)n+pt-nb=qn+pt-nb,$$
$$p\left(ap+t\right)=qp+pt-pb,$$
$$pt-pb<0,~(1<q+pt-pb<q-1)$$ and $$|pt-nb|<q.$$ Furthermore
$$\left\{\frac{p\left(an+t\right)}{q}\right\}=\left\{\frac{pt-nb}{q}\right\}$$
and
$$\left\{\frac{p\left(ap+t\right)}{q}\right\}=\left\{\frac{pt-pb}{q}\right\}=\left\{\frac{q+pt-pb}{q}\right\}=\frac{q+pt-pb}{q}.$$
We deduce that
$$S\left(p,q\right)=\sum_{n=0}^{p-1}\sum_{t=1}^{a}\left\{\frac{an+t}{q}\right\}\left\{\frac{pt-nb}{q}\right\}
+\frac{1}{q^2}\sum_{t=1}^{b-1}\left(ap+t\right)\left(q+pt-pb\right),$$
but
\begin{eqnarray*}
\sum_{n=0}^{p-1}\sum_{t=1}^{a}\left\{\frac{an+t}{q}\right\}\left\{\frac{pt-nb}{q}\right\}
&=&\frac{1}{q^2}\sum_{n=0}^{p-1}\left(\sum_{t<\frac{nb}{p}}\left(an+t\right)\left(q+pt-np\right)+\sum_{t>\frac{nb}{p}}\left(an+t\right)\left(pt-nb\right)\right)\\
&=&\frac{1}{q^2}\sum_{n=0}^{p-1}\left(\sum_{t<\frac{nb}{p}}\left(an+t\right)\left(q+pt-np\right)
-\sum_{t<\frac{nb}{p}}\left(an+t\right)\left(pt-nb\right)\right)\\
&+&\frac{1}{q^2}\sum_{n=0}^{p-1}\sum_{t=1}^{a}\left(an+t\right)\left(pt-nb\right)\\
&=&\frac{1}{q^2}\sum_{n=0}^{p-1}\sum_{t=1}^{a}\left(an+t\right)\left(pt-nb\right)+\frac{1}{q}\sum_{n=0}^{p-1}\sum_{t<\frac{nb}{p}}\left(an+t\right).
\end{eqnarray*}
Then the result follows
\end{proof}
The following lemma computes the three sums in the relation
\eqref{equa15} of the Proposition \ref{propo1}.
\begin{lemma}\label{lemma4}
\begin{eqnarray}\label{equa16}
\sum_{n=0}^{p-1}\sum_{t=1}^{a}\left(an+t\right)
\left(tp-nb\right)&=&-\frac{1}{6p}\left(q-b\right)^2b\left(p-1\right)\left(2p-1\right)\\
\nonumber&+&\frac{1}{4p}\left(q-2b\right)\left(q-b\right)\left(p-1\right)\left(p+q-b\right)\\
\nonumber&+&\frac{1}{6p}\left(q-b\right)\left(p+q-b\right)\left(2q+p-2b\right)
\end{eqnarray}
\begin{eqnarray}\label{equa17}\\
\nonumber
\sum_{n=0}^{p-1}\sum_{t<\frac{nb}{p}}\left(an+t\right)=\frac{1}{12p}\left(p-1\right)\left(2p-1\right)\left(2qb-b^2+1\right)+
\frac{\left(p-1\right)\left(b-1\right)}{4}-qS\left(q,p\right)
\end{eqnarray}
\begin{eqnarray}\label{equa18}\\
\nonumber
\sum_{t=1}^{b-1}\left(ap+t\right)\left(q+pt-pb\right)=\left(q-b\right)\left(q-pb\right)\left(b-1\right)+\frac{b\left(qp+q-2pb\right)\left(b-1\right)}{2}+
\frac{pb\left(b-1\right)\left(2b-1\right)}{6}
\end{eqnarray}
\end{lemma}
\begin{proof}
For the relation \eqref{equa16} we have
$$\left(an+t\right)\left(tp-nb\right)=-abn^2+\left(pa-b\right)nt=pt^2=-abn^2+\left(q-2b\right)nt+pt^2$$
and
$$\sum_{t=1}^{a}t=\frac{a\left(a+1\right)}{2},$$
$$\sum_{t=1}^{a}t^2=\frac{a\left(a+1\right)\left(2a+1\right)}{6}.$$
Thus
$$\sum_{t=1}^{a}\left(an+t\right)\left(tp-nb\right)=-ba^2n^2+\frac{a\left(a+1\right)\left(q-2b\right)n}{2}+\frac{ap\left(a+1\right)\left(2a+1\right)}{6},$$
and
\begin{eqnarray*}
\sum_{n=0}^{p-1}\sum_{t=1}^{a}\left(an+t\right)\left(tp-nb\right)&=&-\frac{ba^2p\left(p-1\right)\left(2p-1\right)}{6}\\
&+&\frac{ap\left(a+1\right)\left(q-2b\right)\left(p-1\right)}{4}\\
&+&\frac{ap^2\left(a+1\right)\left(2a+1\right)}{6}
\end{eqnarray*}
then
\begin{eqnarray*}
\sum_{n=0}^{p-1}\sum_{t=1}^{a}\left(an+t\right)\left(tp-nb\right)&=&-\frac{b\left(ap\right)^2\left(p-1\right)\left(2p-1\right)}{6p}\\
&+&\frac{ap\left(ap+p\right)\left(q-2b\right)\left(p-1\right)}{4p}\\
&+&\frac{ap\left(ap+p\right)\left(2ap+p\right)}{6p}
\end{eqnarray*}
Since $ap=q-b$,$ap+p=q+p-b$ and $2ap+p=2q+p-2b$, we obtain
\begin{eqnarray*}
\sum_{n=0}^{p-1}\sum_{t=1}^{a}\left(an+t\right)\left(tp-nb\right)&=&-\frac{b\left(q-b\right)^2\left(p-1\right)\left(2p-1\right)}{6p}\\
&+&\frac{\left(q-b\right)\left(q+p-b\right)\left(q-2b\right)\left(p-1\right)}{4p}\\
&+&\frac{\left(q-b\right)\left(q+p-b\right)\left(2q+p-2b\right)}{6p}
\end{eqnarray*}

For the second relation \eqref{equa17} we have
\begin{eqnarray*}
\sum_{n=0}^{p-1}\sum_{t<\frac{nb}{p}}\left(an+t\right)&=&a\sum_{n=0}^{p-1}n\left\lfloor\frac{nb}{p}\right\rfloor+\sum_{n=0}^{p-1}\sum_{t<\frac{nb}{p}}t\\
&=&a\sum_{n=0}^{p-1}n\left\lfloor\frac{nb}{p}\right\rfloor+
\frac{1}{2}\sum_{n=0}^{p-1}\left\lfloor\frac{nb}{p}\right\rfloor\left(\left\lfloor\frac{nb}{p}\right\rfloor+1\right)\\
\end{eqnarray*}
But
\begin{eqnarray*}
\sum_{n=0}^{p-1}\left\lfloor\frac{nb}{p}\right\rfloor\left(\left\lfloor\frac{nb}{p}\right\rfloor+1\right)&=&\sum_{n=0}^{p-1}\left\lfloor\frac{nb}{p}\right\rfloor
+\sum_{n=0}^{p-1}\left\lfloor\frac{nb}{p}\right\rfloor^2\\
&=&\frac{\left(p-1\right)\left(b-1\right)}{2}+\frac{\left(b^2+1\right)\left(p-1\right)\left(2p-1\right)}{6p}-2bS\left(b,p\right)
\end{eqnarray*}
and
\begin{eqnarray*}
\sum_{n=0}^{p-1}n\left\lfloor\frac{nb}{p}\right\rfloor&=&\frac{b\left(p-1\right)\left(2p-1\right)}{6}-pS\left(b,p\right)
\end{eqnarray*}
Since
$$\sum_{n=0}^{p-1}\sum_{t<\frac{nb}{p}}t=\frac{1}{2}\sum_{n=0}^{p-1}\left\lfloor\frac{nb}{p}\right\rfloor\left(\left\lfloor\frac{nb}{p}\right\rfloor+1\right)$$
and $$S\left(b,p\right)=S\left(q,p\right)$$ then
$$\sum_{n=0}^{p-1}\sum_{t<\frac{nb}{p}}\left(an+t\right)=\frac{1}{12p}\left(p-1\right)\left(2p-1\right)\left(2qb-b^2+1\right)+
\frac{\left(p-1\right)\left(b-1\right)}{4}-qS\left(q,p\right)$$

For the last one \eqref{equa18}, we have
\begin{eqnarray*}
\sum_{t=1}^{b-1}\left(ap+t\right)\left(q+pt-pb\right)&=&\sum_{t=1}^{b-1}\left(ap\left(q-pb\right)+\left(ap^2+q-pb\right)t+pt^2\right)\\
&=&\sum_{t=1}^{b-1}\left(\left(q-b\right)\left(q-pb\right)+\left(qp+q-2pb\right)t+pt^2\right)\\
&=&\left(q-b\right)\left(q-pb\right)\left(b-1\right)+\frac{b\left(qp+q-2pb\right)\left(b-1\right)}{2}+\frac{pb\left(b-1\right)\left(2b-1\right)}{6}
\end{eqnarray*}
\end{proof}
The substitution of the relations \eqref{equa16}, \eqref{equa17} and
\eqref{equa18} of Lemma \ref{lemma4} in the relation \eqref{equa15}
of the proposition \ref{propo1} conduct to the following result.
\begin{corollary}\label{coro3}
\begin{eqnarray}\label{equa0}\\
\nonumber S\left(p,q\right)+S\left(q,p\right)&=&-\frac{b\left(q-b\right)^2\left(p-1\right)\left(2p-1\right)}{6pq^2}\\
\nonumber&+&\frac{\left(q-b\right)\left(q+p-b\right)\left(q-2b\right)\left(p-1\right)}{4pq^2}\\
\nonumber&+&\frac{\left(q-b\right)\left(q+p-b\right)\left(2q+p-2b\right)}{6pq^2}\\
\nonumber&+&\frac{\left(q-b\right)\left(q-pb\right)\left(b-1\right)}{q^2}+\frac{b\left(qp+q-2pb\right)\left(b-1\right)}{2q^2}+
\frac{pb\left(b-1\right)\left(2b-1\right)}{6q^2}\\
\nonumber&+&\frac{1}{12pq}\left(p-1\right)\left(2p-1\right)\left(2qb-b^2+1\right)+
\frac{\left(p-1\right)\left(b-1\right)}{4q}
\end{eqnarray}
\end{corollary}
\subsection{Proof of Theorem \ref{premiertheo}}
In the case  {\bf b=1}. Taking $b=1$ in the relation \eqref{equa16}
Lemma \ref{lemma4} we obtain
\begin{eqnarray*}
\sum_{n=0}^{p-1}\sum_{t=1}^{a}\left(an+t\right)
\left(tp-n\right)&=&-\frac{1}{6p}\left(q-1\right)^2\left(p-1\right)\left(2p-1\right)\\
\nonumber&+&\frac{1}{4p}\left(q-2\right)\left(q-1\right)\left(p-1\right)\left(p+q-1\right)\\
\nonumber&+&\frac{1}{6p}\left(q-1\right)\left(p+q-1\right)\left(2q+p-2\right),
\end{eqnarray*}
and then
\begin{eqnarray*}
S\left(p,q\right)
&=&-\frac{1}{6pq^2}\left(q-1\right)^2\left(p-1\right)\left(2p-1\right)\\
\nonumber&+&\frac{1}{4pq^2}\left(q-2\right)\left(q-1\right)\left(p-1\right)\left(p+q-1\right)\\
\nonumber&+&\frac{1}{6pq^2}\left(q-1\right)\left(p+q-1\right)\left(2q+p-2\right).
\end{eqnarray*}
Since
$$S\left(q,p\right)=S\left(1,p\right)$$ and from the relation
\eqref{digamma1} Corollary \ref{coro2}, we deduce that
$$S\left(q,p\right)=\frac{\left(p-1\right)\left(2p-1\right)}{6p}.$$
Then
\begin{eqnarray}\label{equa19}\\
\nonumber S\left(p,q\right)+S\left(q,p\right)
\nonumber&=&\frac{1}{6pq^2}\left(p-1\right)\left(2p-1\right)\left(2q-1\right)\\
\nonumber&+&\frac{1}{4pq^2}\left(q-2\right)\left(q-1\right)\left(p-1\right)\left(p+q-1\right)\\
\nonumber&+&\frac{1}{6pq^2}\left(q-1\right)\left(p+q-1\right)\left(2q+p-2\right).
\end{eqnarray}
But
\begin{eqnarray}\label{equa20}\\
\nonumber\left(2q-1\right)\left(p-1\right)\left(2p-1\right)=4p^2q-2p^2-6pq+3p+2q-1
\end{eqnarray}
\begin{eqnarray}\label{equa21}\\
\nonumber\left(q-2\right)\left(q-1\right)\left(p-1\right)\left(p+q-1\right)=p^2q^2+pq^3-3p^2q-5pq^2-q^3+4q^2+2p^2+8pq-4p-5q+2
\end{eqnarray}
\begin{eqnarray}\label{equa22}\\
\nonumber
\left(q-1\right)\left(p+q-1\right)\left(2q+p-2\right)=2q^3+p^2q+3pq^2-6q^2-p^2-6pq+3p+6q-2
\end{eqnarray}

Substitute the relations \eqref{equa20}, \eqref{equa21} and \eqref{equa22} in the relation \eqref{equa19} we get the result.\\

Case {\bf $b\geq2$}\\

The decomposition of the elements of the expression \eqref{equa0} of
$S\left(p,q\right)+S\left(q,p\right)$ in Corollary \ref{coro3},
conduct to

$$b\left(q-b\right)^2=b^3-2qb^2+q^2b$$
$$(q-b)(p+q-b)(q-2b)=-2b^3+(2p+5q)b^2-(3pq+4q^2)b+pq^2+q^3$$
$$(q-b)(p+q-b)(2q+p-2b)=-2b^3+(3p+6q)b^2-(p^2+6pq+6q^2)b+p^2q+3pq^2+2q^3$$
$$(q-b)(q-pb)(b-1)=pb^3-(pq+q+p)b^2+(q^2+pq+q)b-q^2$$
$$b(b-1)(pq+q-2pb)=-2pb^3+(pq+2p+q)b^2-(q+pq)b$$
\begin{eqnarray*}
b(b-1)(2b-1)&=& 2b^3-3b^2+b.
\end{eqnarray*}
Then $S\left(p,q\right)+S\left(q,p\right)$ is the polynomial
$$P(b)=a_0b^3+a_1b^2+a_2b+a_3$$ of degree $3$; and its coefficients are
$$a_0=-\frac{\left(p-1\right)\left(2p-1\right)}{6pq^2}-\frac{p-1}{2pq^2}-\frac{1}{3pq^2}+\frac{p}{q^2}-\frac{p}{q^2}+\frac{p}{3q^2}$$
$$a_1=\frac{q\left(p-1\right)\left(2p-1\right)}{3pq}+\frac{\left(p-1\right)\left(2p+5q\right)}{4pq^2}+\frac{3p+6q}{2pq^2}-\frac{pq+p+q}{q^2}
+\frac{pq+2p+q}{2q^2}-\frac{p}{2q^2}-\frac{\left(p-1\right)\left(2p-1\right)}{12pq}$$
\begin{eqnarray*}
a_2&=&-\frac{(p-1)(2p-1)q}{6pq}-\frac{(p-1)(3pq+4q^2)}{4pq^2}-\frac{p^2+6pq+6q^2}{6pq^2}+\frac{q^2+pq+q}{q^2}-\frac{q+pq}{2q^2}\\
&+&\frac{p}{6q^2}+\frac{(p-1)(2p-1)q}{6pq}
\end{eqnarray*}

\begin{eqnarray*}
a_3&=&\frac{(p-1)(pq^2+q^3)}{4pq^2}+\frac{p^2q+3pq^2+2q^3}{6pq^2}-1+\frac{(p-1)(2p-1)}{12pq}-\frac{p-1}{4q}\\
&=&\frac{1}{12pq}\left[3(p-1)(pq+q^2)+2(p^2+3pq+2q^2)+(p-1)(2p-1)-3p(p-1)\right]-1.
\end{eqnarray*}

Then we obtain $$a_0=a_1=a_2=0$$ and
$$a_3=\frac{p^2+q^2+1}{12pq}+\frac{p+q}{4}-\frac{3}{4}$$
Furthermore the result follows.
\subsection{Proof of Theorem \ref{deuxiemetheo} and Corollary
\ref{coro1}} To prove the Theorem \ref{deuxiemetheo} we need the
following lemma
\begin{lemma}\label{lemma5}
\begin{eqnarray}\label{equa23}
S(p,2)=\frac{1}{4}
\end{eqnarray}
\begin{eqnarray}\label{equa24}
S(p,3)=\frac{1}{3}\left(2-\left\{\frac{p}{3}\right\}\right)
\end{eqnarray}
\begin{eqnarray}\label{equa25}
S\left(p,4\right)=1-\frac{1}{2}\left\{\frac{p}{4}\right\}
\end{eqnarray}
\end{lemma}
\begin{proof}
The first relation \eqref{equa23} is trivial.\\
For the others, first we remark that if $q\equiv b[p]$ then $q=ap+b$
and $\frac{q}{p}=a+\frac{b}{p}$, since $0\leq b\leq p-1$ we deduce
that $\left\lfloor \frac{q}{p}\right\rfloor=a$ and then
$b=q-p\left\lfloor \frac{q}{p}\right\rfloor$ thus
$$b=p\left\{\frac{q}{p}\right\}$$.

For the second relation \eqref{equa24}, since $p\equiv1~or~2[3]$, we
obtain
$$S\left(1,3\right)=\frac{5}{9}~\text{and}~ S\left(2,3\right)=\frac{4}{9}.$$
We remark for $b\in\left\{1,2\right\}$ that
$$S\left(p,3\right)=S\left(b,3\right)=\frac{6-b}{9}$$, since $$b=3\left\{\frac{p}{3}\right\}$$ then
$$S\left(p,3\right)=\frac{2}{3}-\frac{1}{3}\left\{\frac{p}{3}\right\}$$
For the third relation \eqref{equa25}, since $p\equiv1~or~3[4]$, we
obtain
$$S\left(1,4\right)=\frac{7}{8}~\text{and}~ S\left(3,4\right)=\frac{5}{8}.$$
We remark for $b\in\left\{1,3\right\}$ that
$$S\left(p,4\right)=S\left(b,4\right)=\frac{8-b}{8}$$ and then
$$S\left(p,4\right)=1-\frac{1}{2}\left\{\frac{p}{4}\right\}$$
\end{proof}
\begin{corollary}\label{coro4}
Let $p$ any integer, then we have\\
if $\left(p,2\right)=1$ then
\begin{eqnarray}\label{equa26}
S(2,p)=\frac{7p}{24}+\frac{5}{24p}-\frac{1}{2},
\end{eqnarray}
if $\left(p,3\right)=1$ then
\begin{eqnarray}\label{equa27}
S(3,p)=\frac{5p}{18}+\frac{5}{18p}+\frac{1}{3}\left\{\frac{p}{3}\right\}-\frac{2}{3},
\end{eqnarray}
and if $\left(p,4\right)=1$:
\begin{eqnarray}\label{equa28}
S(4,p)=\frac{13p}{48}+\frac{17}{48p}+\frac{1}{2}\left\{\frac{p}{4}\right\}-\frac{3}{4}
\end{eqnarray}
\end{corollary}
\begin{proof}
For the first relation \eqref{equa26}, applying the reciprocity law
\eqref{loiricipro} for $p$ and $2$ we get
$$S\left(p,2\right)+S\left(2,p\right)=\frac{p^2+5}{24p}+\frac{p+2}{4}-\frac{3}{4}$$
From the relation \eqref{equa23} Lemma \ref{lemma5} we deduce that
\begin{eqnarray*}
S\left(2,p\right)&=&\frac{p^2+5}{24p}+\frac{p+2}{4}-\frac{3}{4}-\frac{1}{4}\\
&=&\frac{7p}{24}+\frac{5}{24p}-\frac{1}{2}
\end{eqnarray*}
We do the same thing for the second and the third relation.
\end{proof}
\subsubsection{Proof of Theorem \ref{deuxiemetheo}}
The relation \eqref{loiricipro1} is the consequence of the
reciprocity law \eqref{loiricipro} and the expression
\eqref{digamma1} of $S\left(1,q\right)$ in Corollary \ref{coro2}.\\
To obtain the other relations we must combine the reciprocity law
\eqref{loiricipro} and the respective results in Lemma \ref{lemma5}.

\subsubsection{Proof of the Corollary \ref{coro1}}
The corollary \ref{coro1} is the consequence of the Theorem
\ref{deuxiemetheo} and the Corollary \ref{coro4}.

\end{document}